\documentclass[leqno,onefignum,onetabnum]{siamart0516}
\usepackage{url}

\usepackage{comment}
\usepackage{graphicx}
\usepackage{amsmath,amssymb,amsfonts,amstext,mathrsfs}
\usepackage{thmtools}
\usepackage{thm-restate}
\usepackage{latexsym,epstopdf}
\graphicspath{{./figs/}}
\usepackage{cleveref}
\crefname{assum}{assumption}{assumptions}
\Crefname{assumption}{Assumption}{Assumptions}
\allowdisplaybreaks

\usepackage{dsfont}

\usepackage{algorithm}
\usepackage{algorithmicx}
\usepackage{algpseudocode}
\usepackage{caption}
\usepackage{subcaption}
\usepackage{color} 
\usepackage{todonotes}
\usepackage{enumitem}

\newtheorem{assum}{Assumption}
\newtheorem{remark}{Remark}

\newcommand{\pap}{PAP\xspace}
\usepackage{xspace}

\newcommand{\eg}{{e.g.}\xspace}
\newcommand{\ie}{{i.e.}\xspace}

\newcommand{\cf}{{cf.}\xspace}
\newcommand{\wrt}{{w.r.t.}\xspace}
\newcommand{\aka}{{a.k.a.}\xspace}

\newcommand{\ab}{\mathbf{a}}

\newcommand{\xb}{\mathbf{x}}
\newcommand{\yb}{\mathbf{y}}
\newcommand{\zb}{\mathbf{z}}

\newcommand{\ub}{\mathbf{u}}

\newcommand{\RR}{\mathds{R}}
\newcommand{\dist}{\mathrm{dist}}
\newcommand{\prox}[2][\eta]{\mathrm{prox}_{#2}^{#1}}
\newcommand{\proj}[1]{\mathrm{proj}_{#1}}
\newcommand{\dom}{\mathop{\mathrm{dom}}}
\newcommand{\zero}{\mathbf{0}}
\newcommand{\crit}{\mathop{\mathrm{crit}}}
\newcommand{\argmin}{\mathop{\mathrm{argmin}}}
\newcommand{\KL}{K\L\xspace}
\newcommand{\inner}[2]{\langle #1, #2 \rangle}
\newcommand{\ovset}[2]{\overset{\mathrm{#1}}{#2}}
\newcommand{\dind}{k}
\newcommand{\mspg}{\textsf{m-PAPG}\xspace}

\title{Distributed Proximal Gradient Algorithm for Partially Asynchronous Computer Clusters}

\author{Yi Zhou\thanks{Syracuse University. (\email{yzhou35@syr.edu}, \email{yliang06@syr.edu}).}  
   \and  Yaoliang Yu\thanks{University of Waterloo. (\email{yaoliang.yu@uwaterloo.ca}).} 
   \and  Wei Dai\thanks{Carnegie Mellon University. (\email{wdai@cs.cmu.edu}, \email{epxing@cs.cmu.edu}).} 
   \and  Yingbin Liang%
   \footnotemark[1]
   \and  Eric P. Xing%
   \footnotemark[3]
   \thanks{The material in this paper is presented in part at the 19th International Conference on Artificial Intelligence and Statistics (AISTATS), Cadiz, Spain, 2016.
   }  
}

\begin{document}
\maketitle

\begin{abstract}
With ever growing data volume and model size, an error-tolerant, communication efficient, yet versatile distributed algorithm has become vital for the success of many large-scale machine learning applications. In this work we propose \mspg, an implementation of the flexible proximal gradient algorithm in model parallel systems equipped with the partially asynchronous communication protocol. The worker machines communicate asynchronously with a controlled staleness bound $s$ and operate at different frequencies. We characterize various convergence properties of \mspg: 1) Under a general non-smooth and non-convex setting, we prove that every limit point of the sequence generated by \mspg is a critical point of the objective function; 2) Under an error bound condition, we prove that the function value decays linearly for every $s$ steps; 3) Under the Kurdyka-${\L}$ojasiewicz inequality, we prove that the sequences generated by \mspg converge to the same critical point, provided that a proximal Lipschitz condition is satisfied.
\end{abstract}

\begin{AMS}\end{AMS}

\pagestyle{myheadings}
\thispagestyle{plain}
\markboth{Zhou, Yu, Dai, Liang, and Xing}{Distributed Proximal Gradient Algorithm for Partially Asynchronous Computer Clusters}


\section{Introduction}\label{sec: intro}
The composite minimization problem
\begin{align}\label{com_min}
\min_{\xb\in \RR^d}~ f(\xb) + g(\xb)
\end{align}
has drawn a lot of recent attention due to its ubiquity in machine learning and statistical applications.
Typically, the first term 
\begin{align}\label{finitesum}
f(\xb) := \frac{1}{n}\sum_{i=1}^n f_i(\xb)
\end{align} 
is a smooth loss function over $n$ training samples that describes the fitness to data, and the second term $g$ is a nonsmooth regularization function that encodes \emph{a priori} information. We list below some popular examples under this framework.
\begin{itemize}[leftmargin=20pt,topsep=3pt,itemsep=3pt]
\item Lasso: least squares loss $f_i(\xb) = (y_i - \ab_i^\top \xb)^2$ and $\ell_1$ norm regularizer $g(\xb) = \|\xb\|_1$;
\item Logistic regression: logistic loss $f_i = \log(1+\exp(-y_i \ab_i^\top \xb_i))$;
\item Boosting: exponential loss $f_i(\xb) = \exp(-y_i \ab_i^\top\xb)$;
\item Support vector machines: hinge loss $f_i(\xb) =\max\{0, 1-y_i\ab_i^\top \xb\}$ and (squared) $\ell_2$ norm regularizer $g(\xb) = \|\xb\|_2^2$.
\end{itemize}
Over the years there is also a rising interest in using nonconvex losses $f$ (mainly for robustness against outlying observations) \cite{Collobert06,WuLiu07,Xuetal06,YuZMX15} and nonconvex regularizers $g$ (mainly for smaller bias in statistical estimation) \cite{FanLi01,ZhangZhang12}.

Due to the apparent importance of the composite minimization framework and the rapidly growing
size in both dimension ($d$) and volume ($n$) of data, there is a strong need to develop a practical \emph{parallel} system that can solve the problem in \eqref{com_min} efficiently and in a scale that is impossible for a single machine \cite{distrib_delay,BT_parallel, hadoop, delay_pga_linear, ssp_parameter, parameter_server,graphlab, spark}.
Existing systems can be categorized by how communication among worker machines is managed: bulk synchronous (also called fully synchronous) \cite{hadoop,bsp,spark}, totally asynchronous \cite{Baudet78,BT_parallel,graphlab}, and partially asynchronous (\aka stale synchronous or chaotic) \cite{distrib_delay,BT_parallel,ChazanMiranker69,delay_pga_linear,ssp_parameter,parameter_server,Tseng_linear_PAA}. Bulk synchronous parallel (BSP) systems explicitly force synchronization barriers so that the worker machines can stay on the same page to ensure correctness.
However, in a real deployed parallel system, BSP usually suffers from the straggler problem, that is, the performance of the whole system is bottlenecked at the bandwidth of communication and the \emph{slowest} worker machine. 
On the other hand, totally asynchronous systems do not put any constraint on synchronization, hence achieve much greater throughputs by potentially sacrificing the correctness of the algorithm. 
Partially asynchronous parallel (\pap) systems \cite{BT_parallel,ChazanMiranker69} are a compromise between the previous two: it allows the worker machines to communicate asynchronously up to a controlled staleness and to perform updates at different paces. \pap is particularly suitable for machine learning applications, where iterative algorithms that are robust to small computational errors are usually favored for finding an appropriate solution. Due to its flexibility, the \pap mechanism has been the method of choice in many recent practical implementations \cite{distrib_delay,delay_pga_linear,ssp_parameter,parameter_server,asy_coord_prox,hogwild}.

Existing parallel systems can also be categorized by how computation is divided among worker machines: data parallel and model parallel. Data parallel systems usually distribute the computation involving each component function $f_i$ in \eqref{finitesum} into different worker machines, which is suitable when $n \gg d$, \ie, large data volume but moderate model size. In this setting the stochastic proximal gradient algorithm, along with the \pap protocol, has been shown to be quite effective in solving the composite problem \eqref{com_min} \cite{distrib_delay,delay_pga_linear,ssp_parameter,parameter_server}. In this work, we focus on the ``dual'' model parallel regime where $d \gg n$, \ie, large model size but moderate data volume. 
In modern machine learning and statistics applications, it is not uncommon that the dimensionality of data largely exceeds its volume, for example, in computational biology, conducting an experimental study that involves many patients can be very expensive but for each patient, technology (\eg next-generation genome sequencing) has advanced to a stage where taking a large number of measurements (model parameters) is relatively cheap. Deep neural networks are another example that calls for model parallelism.
Not surprisingly, the design of a model parallel system is fundamentally different from that of a data parallel system, and so is the subsequent analysis. 

To achieve model parallelism, the model $\xb$ is partitioned into different (disjoint) blocks and is distributed among many worker machines. In this setting, 
the block proximal gradient algorithm has been proposed to solve the composite problem \eqref{com_min} \cite{FercoqRichtarik15, LuXiao15, random_block_coord}, although under the more restrictive BSP protocol.
Under the \pap protocol, the only work that we are aware of is \cite{BT_parallel} which focused on a special case of \eqref{com_min} where $g$ is an indicator function of a convex set, and \cite{Tseng_linear_PAA} which established a periodic linear rate of convergence under an error bound condition. 
Our main goal in this work is to provide a formal convergence analysis of the model parallel  proximal gradient algorithm under the more flexible \pap communication protocol, and our results naturally extend those in \cite{BT_parallel,Tseng_linear_PAA} to allow nonsmooth and nonconvex functions.

Our main contributions in this work are: 1). We propose \mspg, an extension of the proximal gradient algorithm to the model parallel and partially asynchronous setting.
2). We provide a rigorous analysis of the convergence properties of \mspg, allowing both \emph{nonsmooth} and \emph{nonconvex} functions. In particular, we prove in \Cref{thm:limitcon_1} that any limit point of the sequences generated by \mspg is a critical point. 
3) Under an additional error bound condition, we prove in \Cref{coro: s step linear} that the function values generated by \mspg decays periodically linearly.
4) Lastly, using the Kurdyka-${\L}$ojasiewicz (\KL) inequality \cite{prox_alter_linear}, we prove in \Cref{thm:finite} that for functions that satisfy a proximal Lipschitz condition the whole sequences of \mspg converge to a single critical point.

This paper proceeds as follows: We first %
set up the notations and definitions in \Cref{sec: Pre}. The proposed algorithm \mspg is presented in \Cref{sec: models}, and convergence analysis are detailed in \Cref{sec: limit point,sec: EB,sec: KL}. \Cref{sec: conclude} concludes our work.


\section{Preliminaries}\label{sec: Pre}

We first recall some fundamental definitions that will be needed in our analysis. Throughout, $h:\RR^d \to (-\infty, +\infty]$ denotes an extended real-valued function that is proper and closed, \ie, its domain $\dom h := \{ \xb: h(\xb) < +\infty \}$ is nonempty and its sublevel set $\{ \xb: h(\xb) \leq \alpha \}$ is closed for all $\alpha\in\RR$. Since the function $h$ may not be smooth or convex, we need the following generalized notion of ``derivative.''
\begin{definition}[Subdifferential and critical point, \eg \cite{vari_ana}]
\label{def:sub}
The Frech\'et subdifferential $\hat\partial h$ of $h$ at $\xb\in\dom h$ is the set of $\ub$ such that
\begin{align}
\liminf_{\zb\neq\xb, \zb\to\xb} \frac{h(\zb) - h(\xb) - \ub^\top(\zb-\xb)}{\|\zb-\xb\|} \geq 0,
\end{align}
while the (limiting) subdifferential $\partial h$ at $\xb\in\dom h$ is the ``closure'' of $\hat\partial h$:
\begin{align}
\{ \ub: \exists \xb^k \to \xb, h(\xb^k) \to h(\xb), \ub^k \in \hat{\partial} h(\xb^k), \ub^k \to \ub \}.
\end{align}
The critical points of $h$ are $\crit h := \{ \xb: \zero\in\partial h(\xb) \}$. 
\end{definition}

When $h$ is continuously differentiable or convex, the subdifferential $\partial h$ and the set of critical points $\crit h$ coincide with the usual notions. For a closed function $h$, its subdifferential is either nonempty at any point in its domain or the subgradient diverges to some ``direction'' \cite[Corollary 8.10]{vari_ana}.
\begin{definition}[Distance and projection]\label{def: distance}
The distance function \wrt a closed set $\Omega \subseteq \RR^d$ is defined as:
\begin{align}
\dist_{\Omega}(\xb) := \min_{\yb\in\Omega} \|\yb-\xb\|,  
\end{align}
while the metric projection onto $\Omega$ is defined  as:
\begin{align}
\proj{\Omega}(\xb) := \argmin_{\yb\in\Omega} \|\yb-\xb\|,
\end{align}
where $\|\cdot\|$ is the usual Euclidean norm.
\end{definition}

Note that $\proj{\Omega}(\xb)$ is single-valued for all $\xb\in\RR^d$ if and only if $\Omega$ is convex.

\begin{definition}[Proximal map, \eg \cite{vari_ana}]\label{def:prox_map}
The proximal map of a closed and proper function $h$ is (with parameter $\eta > 0$):
\begin{align}
\prox{h}(\xb) := \argmin_{\zb\in \RR^d} h(\zb) + \tfrac{1}{2\eta}\|\zb - \xb\|^2.
\end{align}
Occasionally, we will write $\prox[]{h}$ instead of $\prox[1]{h}$.
\end{definition}

Clearly, for the indicator function
$
h(\xb) = \iota_\Omega(\xb)$, which takes the value 0 for $\xb\in \Omega$ and $\infty$ otherwise, 
its proximal map (with any $\eta > 0$) reduces to the metric projection $\proj{\Omega}$. 
If $h$ decreases slower than a quadratic function (in particular, when $h$ is bounded below), then its proximal map is well-defined for all (small) $\eta$ \cite{vari_ana}. If $h$ is convex, then its proximal map is always a singleton while for nonconvex $h$, the proximal map can be set-valued. In the latter case we will also abuse the notation $\prox{h}(\xb)$ for an arbitrary element from that set. For convex functions, the proximal map is nonexpansive:
\begin{align}
\label{eq:ne}
\forall \xb, \forall \yb, ~ \|\prox{h}(\xb) - \prox{h}(\yb) \| \leq \|\xb - \yb\|,
\end{align}
while for nonconvex functions this may not hold everywhere.

The proximal map is the key component of the proximal gradient algorithm \cite{FukushimaMine81} (\aka forward-backward splitting):
\begin{align}
\label{eq:pg}
\forall\,t = 0, 1, \ldots, \quad \xb(t+1) = \prox[\eta]{g}\big(\xb(t) - \eta \nabla f(\xb(t))\big),
\end{align}
where $\nabla f$ is the (sub)gradient of $f$, and $\eta$ is a suitable step size (that may  change with $t$). It is known that when $f$ is convex with $L$-Lipschitz continuous gradient  and $0 < \eta < 2/L$, then $F_t := f(\xb(t))+g(\xb(t))$ converges to the minimum at the rate $O(1/t)$ and $\xb(t)$ converges to some minimizer $\xb^*$. Accelerated versions \cite{FISTA,Nesterov13} where $F_t$ converges at the faster rate $O(1/t^2)$ are also well-known. Recently, \cite{prox_alter_linear} proved that $\xb(t)$ converges to a critical point even for nonconvex $f$ and nonconvex and nonsmooth $g$ as long as together they satisfy a certain \KL inequality.


\section{Formulation of $\mspg$}\label{sec: models}
Recall the composite minimization problem:
\begin{align}
\label{compmin}
\min_{\xb\in\RR^d} F(\xb), \quad \mbox{ where } \quad F(\xb) = f(\xb) + g(\xb). \tag{P}
\end{align}
We are interested in the case where $d$ is so large that implementing the proximal gradient algorithm \eqref{eq:pg} on a single machine is no longer feasible, hence distributed computation is necessary.

We consider a \textbf{model} parallel system with $p$ machines in total, and decompose the $d$ model parameters into $p$ disjoint groups.
Formally, consider the decomposition $ \RR^d = \RR^{d_1} \times \RR^{d_2} \times\cdots\times\RR^{d_p}$, and denote $x_i$ and $\nabla_i f(\xb): \RR^d \to \RR^{d_i}$ as the $i$-th component of $\xb$ and $\nabla f(\xb)$, respectively. Clearly, $\xb = (x_1, x_2, \ldots, x_p)$ and $\nabla f = (\nabla_1 f, \nabla_2 f, \cdots, \nabla_p f)$. The $i$-th machine is responsible for updating the component $x_i\in\RR^{d_i}$, and for the purpose of evaluating the partial gradient $\nabla_i f(\xb)$ we assume the $i$-th machine also has access to a local, full model parameter $\xb^i\in\RR^d$. The last assumption is made only to simplify our presentation; it can be removed for many machine learning problems, see for instance \cite{random_block_coord,mspg}. 

We make the following standard assumptions regarding problem \eqref{compmin}:
\begin{assum}[Bounded Below]
	\label{assum:bb}
The function $F \!=\! f+g$ is bounded below.
\end{assum}
\begin{assum}[Smooth]
	\label{assum:smooth}
	The gradient $\nabla f$ of $f$ is $L$-Lipschitz continuous:
	\begin{align}
	\forall \xb, \forall \yb, ~ \| \nabla f(\xb) - \nabla f(\yb) \| \leq L \|\xb - \yb\|.
	\end{align}
\end{assum}
\begin{assum}[Separable]
	\label{assum:separable}
	The function $g$ is closed and separable, \ie, $g(\xb) = \sum_{i=1}^{p} g_i(x_i)$.
\end{assum}

\Cref{assum:bb} simply allows us to have a finite minimum value and is usually satisfied in practice.
The smoothness assumption is critical in two aspects: (1) It allows us to upper bound $f$ by its quadratic expansion at the current iterate---a standard step in the convergence proof of gradient type algorithms:
\begin{align}
\label{eq:qb}
\forall \xb, \forall \yb, ~f(\xb) \leq f(\yb) + \langle \xb - \yb, \nabla f(\yb) \rangle + \tfrac{L}{2} \|\xb - \yb\|^2.
\end{align}
 (2) It allows us to bound the inconsistencies in different machines due to asynchronous updates, see \Cref{thm:incon} below. 
The separable assumption is what makes model parallelism interesting and feasible. We remark that both \Cref{assum:smooth} and \Cref{assum:separable} can be relaxed using techniques in \cite{smoothing} and \cite{YuZMX15}, respectively. For brevity we do not pursue these extensions here.
Note that we do \emph{not} assume convexity on either $f$ or $g$, and $g$ need not even be continuous.

We now specify the \mspg algorithm for solving \eqref{compmin} under model parallelism and the \pap protocol. The separable assumption on $g$ implies that 
\begin{align}
\prox{g}(\xb) = \big( \prox{g_1}(x_1),~ \ldots~,~ \prox{g_p}(x_p) \big).
\end{align}
Then, the update on machine $i$ is defined as: 
\begin{align}
\label{update_op}
x_i \gets \prox{g_i}(x_i - \eta \nabla_i f(\xb^i)).
\end{align}
That is, machine $i$ computes a partial gradient mapping \cite{Nesterov13} \wrt the $i$-th component using the local component $x_i$ and the local full model $\xb^i$. To define the latter, consider a global clock shared by all machines and denote $T_i$ as the set of active clocks when machine $i$ performs an update. Note that the global clock is introduced solely for the purpose of our analysis, and the machines need not maintain it in a practical implementation.
Formally, the $t$-th iteration on machine $i$ can be written as:
\begin{equation}\label{mspg_formulation}
\tag{\mspg}
\left\{
\begin{aligned}
	\forall i,~
x_i(t+1) &= \begin{cases}
x_i(t), & t \not\in T_i \\
\prox{g_i}(x_i(t) - \eta \nabla_i f(\xb^i(t))), & t \in T_i
\end{cases}
,
\\ 
\mbox{(local) } \quad \xb^i(t) &= \big( x_1(\tau_1^i(t)),~ \ldots~,~ x_p(\tau_p^i(t)) \big),
\\
\mbox{(global)}\quad  \xb(t) &= \big( x_1(t),~ \ldots~,~ x_p(t) \big).
\end{aligned}
\right.
\end{equation}
That is, machine $i$ only performs its update operator at its active clocks. The local full model $\xb^i(t)$ assembles all components from other machines, and is possibly a delayed version of the global model $\xb(t)$, which assembles the most up-to-date component in each machine. Note that the global model is introduced for our analysis, and is not accessible in a real implementation.
More specifically, $\tau_j^i(t) \leq t$ models the communication delay among machines: when machine $i$ conducts its $t$-th update it only has access to  $x_j(\tau_j^i(t))$, a delayed version of the component $x_j(t)$ on the $j$-th machine. We refer to the above algorithm as \mspg (for \textbf{m}odel parallel, \textbf{P}artially \textbf{A}synchronous, \textbf{P}roximal \textbf{G}radient).

In a practical distributed system, communication among machines is much slower than local computations, and the performance of a \emph{synchronous} system is often bottlenecked at the \emph{slowest} machine, due to the need of synchronization in every step.  
The delays $\tau^i_j(t)$ and active clocks $T_i$ that we introduced in \mspg aim to address such issues. For our convergence proofs, we need the following assumptions:
\begin{assum}[Bounded Delay]\label{assum:delay}
$\exists s \in \mathds{N}, ~\forall i, \forall j, \forall t, ~0 \leq t-\tau_j^i(t) \leq s$, $\tau_i^i(t) \equiv t$.
\end{assum}
\begin{assum}[Frequent Update]\label{assum:skip}
$\exists s \in \mathds{N}, ~\forall i, \forall t, T_i \cap \{t, t+1, \cdots, t+s\} \ne \emptyset$.
\end{assum}

Intuitively, \Cref{assum:delay} guarantees the information that machine $i$ gathered from other machines at the $t$-th iteration are not too obsolete (bounded by at most $s$ clocks apart). The assumption $\tau_i^i(t) \equiv t$ is natural since the $i$-th worker machine is maintaining $x_i$ hence would always have the latest copy.
\Cref{assum:skip}
requires each machine to update at least once in every $s+1$ iterations, for otherwise some component $x_i$ may not be updated at all. We remark that \Cref{assum:delay} and \Cref{assum:skip} are very natural and have been widely adopted in previous works \cite{Baudet78,BT_parallel,ChazanMiranker69,delay_pga_linear,Tseng_linear_PAA}. 
Clearly, when $s = 0$ (\ie, no delay), \mspg reduces to the fully synchronous, model parallel proximal gradient algorithm. 

Before closing this section, we provide a technical tool to control the inconsistency between the local models $\xb^i(t)$ and the global model $\xb(t)$. 
Recall that $(t)_+ = \max\{t, 0\}$ is the positive part of $t$.
\begin{lemma}
\label{thm:incon}
Let \Cref{assum:delay} hold, then the global model $\xb(t)$ and the local models $\{\xb^i(t)\}_{i=1}^p$ satisfy:
\begin{align}
\forall i=1,\cdots,p, ~~ \|\xb(t) - \xb^i(t)\| \leq \sum_{\dind = (t-s)_+}^{t-1} \| \xb(\dind+1) -  \xb(\dind) \| \label{incon_bound}
\\
\label{eq:local_diff}
\|\xb^i(t+1) - \xb^i(t)\| \leq \sum_{\dind = (t-s)_+}^{t} \| \xb(\dind+1) -  \xb(\dind) \|. 
\end{align}
\end{lemma}
\begin{proof}
 	Indeed, by the definitions in \eqref{mspg_formulation}:
 	\begin{align*}
 	\|\xb(t) - \xb^i(t)\|^2 &= \sum_{j=1}^p \|x_j(t) -  x_j(\tau_j^i(t)) \|^2
 	\\
 	 &\leq 
 	\sum_{j=1}^p \left(\sum_{\dind = \tau_j^i(t)}^{t-1} \|x_j(\dind+1) -  x_j(\dind) \|\right)^2
 	\\
 	&{\leq} 
 	\sum_{j=1}^p \left(\sum_{\dind = (t-s)_+}^{t-1} \|x_j(\dind+1) -  x_j(\dind) \|\right)^2
 	\\
 	&=\sum_{j=1}^p \sum_{\dind = (t-s)_+}^{t-1} \sum_{\dind' = (t-s)_+}^{t-1} \|x_j(\dind+1) -  x_j(\dind) \|\|x_j(\dind'+1) -  x_j(\dind') \|
 	\\
&=\sum_{\dind = (t-s)_+}^{t-1} \sum_{\dind' = (t-s)_+}^{t-1} \sum_{j=1}^p \|x_j(\dind+1) -  x_j(\dind) \|\|x_j(\dind'+1) -  x_j(\dind') \| 	
 	\\
&{\leq}\sum_{\dind = (t-s)_+}^{t-1} \sum_{\dind' = (t-s)_+}^{t-1} \|\xb(\dind+1) -  \xb(\dind) \|\|\xb(\dind'+1) -  \xb(\dind') \| 	
 	\\
 	&=
 	\left(\sum_{\dind = (t-s)_+}^{t-1} \| \xb(\dind+1) -  \xb(\dind) \|\right)^2,
 	\end{align*}
 	where the first inequality is due to the triangle inequality; the second inequality is due to \Cref{assum:delay}; and the last inequality follows from the Cauchy-Schwarz inequality.
	
	Similarly,
	\begin{align*}
	\|\xb^{i}(t) - \xb^{i}(t+1)\|^2 &= \sum_{j=1}^p \|x_j(\tau_{j}^i(t)) - x_j(\tau_{j}^i(t+1))\|^2 
	\\ 
	&\leq 
	 \sum_{j=1}^p \left(\sum_{\dind = \tau_j^i(t)}^{\tau_j^i(t+1)-1} \|x_j(\dind+1) -  x_j(\dind) \|\right)^2
	 \\
	 &\leq
	 \sum_{j=1}^p \left(\sum_{\dind = (t-s)_+}^{t} \|x_j(\dind+1) -  x_j(\dind) \|\right)^2,
	\end{align*}
	and the rest of the proof is completely similar to the previous case.
 \end{proof}

\section{Characterizing the limit points}\label{sec: limit point}
In this section, we characterize the convergence property of the sequences generated by \mspg under very general conditions. 
Recall from \Cref{assum:smooth} that $\nabla f$ is $L$-Lipschitz continuous. 
Our first result is as follows:

\begin{theorem}\label{thm:limitcon}
Let \Cref{assum:bb,assum:smooth,assum:separable,assum:delay,assum:skip} hold. If the step  size $\eta \in \left(0, \frac{1}{L(1+2\sqrt{p}s)}\right)$, then the sequence generated by \mspg is square summable, \ie
\begin{align}\label{sq_summable} 
\sum_{t=0}^{\infty} \|\xb(t+1) - \xb(t)\|^2 < \infty.
\end{align}
In particular, $\lim\limits_{t\to\infty} \|\xb(t+1) - \xb(t)\| = 0$ and  $\lim\limits_{t\to\infty} \|\xb(t) - \xb^i(t)\| = 0$.
\end{theorem}
\begin{remark}
Our bound on the step size $\eta$ is natural: If $s = 0$, \ie, there is no asynchronism then we recover the standard step size rule $\eta < 1/L$ (we can increase $\eta$ by another factor of 2, had convexity on $g$ been assumed). As staleness $s$ increases, we need a smaller step size to ``damp'' the system to still ensure convergence. The factor $\sqrt{p}$ is another measurement of the degree of ``dependency'' among worker machines: Indeed, we can reduce $\sqrt{p}$ to $\sqrt{\sum_i L_i^2} / L$, where $L_i$ is the Lipschitz constant of $\nabla_i f$ (\cf \eqref{eq:Lip}). 
\end{remark}
\begin{proof} 
	The last claim follows immediately from \cref{sq_summable} and \cref{incon_bound}, so we only need to prove \eqref{sq_summable}.
	
	Consider machine $i$ and any $t\in T_i$. Combining \cref{update_op} with \cref{mspg_formulation} gives  
	\begin{align}\label{update_exp}
	x_i(t+1) = \prox{g_i}\big(x_i(t) - \eta \nabla_i f(\xb^i(t))\big).
	\end{align}
	Then, from \Cref{def:prox_map} of the proximal map we have for all $z\in\RR^{d_i}$:
	\begin{flalign}\label{usc}
	g_i\big(x_i(t+1)\big) + \frac{1}{2\eta}\|x_i(t+1) &- x_i(t) + \eta\nabla_{i} f\big(\xb^i(t)\big)\|^2 \\
	&\le 
	g_i\big(z\big) + \frac{1}{2\eta}\left\|z - x_i(t) + \eta \nabla_{i} f\big(\xb^i(t)\big)\right\|^2. \nonumber
	\end{flalign}
	Set $z = x_i(t)$ and simplify, we obtain:
	\begin{flalign}
	\label{prog_g}
	g_i\big(x_i(t+1)\big) - &g_i\big(x_i(t)\big)  \\
	&\le 
	- \frac{1}{2\eta}\|x_i(t+1) - x_i(t)\|^2 -\left\langle \nabla_{i} f\big(\xb^i(t)\big),  x_i(t+1) - x_i(t) \right\rangle. \nonumber
	\end{flalign}
	Note that if $t \notin T_i$, then $x_i(t+1) = x_i(t)$ and \cref{prog_g} still holds. 
	On the other hand, \Cref{assum:smooth} implies that for all $t$ (\cf \eqref{eq:qb}):
	\begin{flalign}\label{prog_f}
	f\big(\xb(t+1)\big) - f\big(\xb(t)\big) 
	\le 
	\langle \xb(t+1)-\xb(t), \nabla f\big(\xb(t)\big) \rangle + \frac{L}{2} \|\xb(t+1)-\xb(t)\|^2.
	\end{flalign}
	Adding up \cref{prog_f} and \cref{prog_g} (for all $i$) and recall $F = f+ \sum_i g_i$, we have
	\begin{flalign}
	F\big(\xb(t+1)\big) - F\big(\xb(t)\big) & - \tfrac{1}{2}(L - 1/\eta) \|\xb(t+1) - \xb(t)\|^2 \nonumber\\
	&\le 
	 \sum_{i=1}^{p} \left\langle x_i(t+1)-x_i(t), \nabla_i f(\xb(t)) -  \nabla_{i} f\big(\xb^{i}(t)\big)\right\rangle
	\nonumber\\
	&\le 
	 \sum_{i=1}^{p} \|x_i(t+1)-x_i(t)\| \cdot \|\nabla_i f(\xb(t)) -  \nabla_{i} f\big(\xb^{i}(t)\big) \|
	\nonumber\\
	&\overset{\mathrm{(i)}}{\le} 
	 \sum_{i=1}^{p} \|x_i(t+1)-x_i(t)\| \cdot L \|\xb(t) - \xb^{i}(t) \|
	\label{eq:Lip}\\
	&\overset{\mathrm{(ii)}}{\le} 
	 L \cdot \sum_{i=1}^{p} \|x_i(t+1)-x_i(t)\| \cdot  \sum_{\dind=(t-s)_+}^{t-1} \|\xb(\dind+1) - \xb(\dind) \|
	\nonumber\\
	&\overset{\mathrm{(iii)}}{\le}  
	\sqrt{p} L \|\xb(t+1)-\xb(t)\| \cdot  \sum_{\dind=(t-s)_+}^{t-1} \|\xb(\dind+1) - \xb(\dind) \| \label{eq: func_moment}
	\\
	&\overset{\mathrm{(iv)}}{\le} 
	 \frac{\sqrt{p} L}{2}\sum_{\dind=(t-s)_+}^{t-1} \Big[\|\xb(\dind+1) - \xb(\dind) \|^2 + \|\xb(t+1)-\xb(t)\|^2 \Big]
	\nonumber\\
	&\le  \label{eq: progress}
	\frac{\sqrt{p} L s}{2} \|\xb(t+1) \!-\! \xb(t)\|^2 + \frac{\sqrt{p} L}{2} \!\sum_{\dind=(t-s)_+}^{t-1} \!\!\!\|\xb(\dind+1) \!-\! \xb(\dind) \|^2,
	\end{flalign}
	where (i) is due to the $L$-Lipschitz continuity of $\nabla f$, (ii) follows from \cref{incon_bound}, (iii) is the Cauchy-Schwarz inequality, and (iv) follows from the elementary inequality $ab\leq \tfrac{a^2+b^2}{2}$.
	Summing the above inequality over $t$ from $0$ to $n-1$ and rearranging we obtain
	\begin{flalign*}
	F\big(\xb(n)\big) - F\big(\xb(0)\big) 
	& \leq
	\frac{1}{2}(L + \sqrt{p} Ls - 1/\eta) \sum_{t=0}^{n-1} \|\xb(t+1) - \xb(t)\|^2 \\
	&\quad+ \frac{L}{2} \sum_{t=0}^{n-1} \sum_{\dind=(t-s)_+}^{t-1} \|\xb(\dind+1) - \xb(\dind) \|^2
	\\
	&\leq
	\frac{1}{2}(L + 2\sqrt{p}Ls - 1/\eta)  \sum_{t=0}^{n-1} \|\xb(t+1) - \xb(t)\|^2.
	\end{flalign*}
	Therefore, if we choose $0 < \eta < \frac{1}{L (1 + 2\sqrt{p}s)}$, then let $n\rightarrow \infty$ we deduce
	\begin{align}
	\sum_{t=0}^{\infty} \|\xb(t+1) - \xb(t)\|^2 
	\leq 
	\frac{2}{1/\eta - L - 2\sqrt{p}Ls} [F\big(\xb(0)\big) - \inf_{\zb} F(\zb)].
	\end{align}
	By \Cref{assum:bb}, $F$ is bounded from below, hence the right-hand side is finite. 
\end{proof}

The first assertion of the above theorem states that the global sequence $\xb(t)$ has square summable successive differences, while the second assertion implies that both the successive difference of the global sequence and the inconsistency between the local sequences and the global sequence diminish as the number of iterations grows. These two conclusions provide a prelimenary stability guarantee for \mspg. 

Next, we prove that the limit points (if exist) of the sequences $\xb(t)$ and $\xb^i(t), i=1,\ldots, p$ coincide, and they are critical points of $F$. Again, no convexity assumption is imposed on either $f$ or $g$.

\begin{theorem}\label{thm:limitcon_1}
	Consider the same setting as in \Cref{thm:limitcon}. Then, the sequences $\{\xb(t)\}$ and $\{\xb^i(t)\}, i = 1,\ldots, p$, generated by \mspg share the same set of limit points, which is a subset of $\crit F$.
\end{theorem}

\begin{proof}
	It is clear from \Cref{thm:limitcon} that $\{\xb(t)\}$ and $\{\xb^i(t)\}, i=1, \ldots, p$, share the same set of limit points, and we need to show that any limit point of $\{\xb(t)\}$ is also a critical point of $F$.
	 
	Let $\xb^*$ be a limit point of $\{\xb(t)\}$. By \Cref{def:sub} it suffices to exhibit a sequence $\xb(k)$ satisfying\footnote{Technically, from \Cref{def:sub} we should have the Frech\'et subdifferential $\hat{\partial} F$ in \cref{crit_want}, however, a standard argument allows us to use the more convenient subdifferential \cite[Proposition 8.7]{vari_ana}.} 
	\begin{align}\label{crit_want}
	\xb(k) \to \xb^*, ~ F(\xb(k)) \to F(\xb^*), ~ \zero \leftarrow \ub(k) \in \partial F(\xb(k)).
	\end{align} 
	
	Let us first construct the subgradient sequence $\ub(k)$. Consider machine $i$ and any $\hat t \in T_i$, the optimality condition of \cref{update_exp} gives
	\begin{flalign}
	u_i( \hat t+1) := -\tfrac{1}{\eta} \left[x_i(\hat t+1) - x_i(\hat t) + \eta \nabla_{i} f\big(\xb^i(\hat t)\big)\right] \in \partial g_i(x_i( \hat t+1)).
	\end{flalign}
	It then follows that 
	\begin{flalign}
	\|u_i( \hat t+1) & + \nabla_i f(\xb( \hat t+1))\| \nonumber\\
	&\leq 
	\|u_i( \hat t+1) + \nabla_i f(\xb( \hat t ))\| + \|\nabla_i f(\xb( \hat t+1)) - \nabla_i f(\xb(\hat t))\|
	\nonumber\\
	&\overset{\mathrm{(i)}}{\le} 
	\left\| \tfrac{1}{\eta} \left[ x_i( \hat t+1) - x_i(\hat t) \right] + \nabla_{i} f\big(\xb^{i}(\hat t)\big) - \nabla_{i} f\big(\xb(\hat t)\big) \right\| + L\|\xb(\hat t+1) - \xb(\hat t)\|
	\nonumber\\
	&\overset{\mathrm{(ii)}}{\le} 
	\tfrac{1}{\eta} \| x_i(\hat t+1) - x_i(\hat t) \| + L \|\xb^{i}(\hat t) - \xb(\hat t)\| + L \|\xb(\hat t+1) - \xb(\hat t)\|
	\nonumber\\
	\label{sub_zero_1}
	&\overset{\mathrm{(iii)}}{\le}
	\tfrac{1}{\eta} \| x_i(\hat t+1) - x_i(\hat t) \| + L \sum_{\dind=(\hat t-s)_+}^{\hat t} \|\xb(\dind+1) - \xb(\dind)\|,
	\end{flalign}
	where (i) and (ii) are due to the $L$-Lipschitz continuity of $\nabla f$, and (iii) follows from \cref{incon_bound}.
	Next, consider any other $t\not\in T_i$ and $t\geq s$, we denote $\hat t$ as the \emph{largest} element in the set $\{ \dind \leq t: \dind \in T_i \}$. By \Cref{assum:skip} $\hat t$ always exists and $t - \hat t\leq s$. Since no update is performed on machine $i$ at any clock in $[\hat{t}+1, t]$, we have $x_i(t+1) = x_i(\hat t + 1)$. Thus, we can choose $u_i(t+1) = u_i(\hat t+1) \in \partial g_i (x_i(\hat{t}+1)) = \partial g_i (x_i(t+1))$, and obtain
	\begin{align}
	\| u_i(t+1) + \nabla_i f(\xb(t+1)) - & u_i(\hat t + 1) - \nabla_i f(\xb(\hat t + 1)) \|  \\
	&= 
	\| \nabla_i f(\xb(t+1)) - \nabla_i f(\xb(\hat t + 1)) \|
	\nonumber\\
    &\leq
	\sum_{\dind = \hat t
		+1}^t \|\nabla_i f(\xb(\dind+1)) - \nabla_i f(\xb(\dind))\|
	\nonumber\\
	&\leq 
	\sum_{\dind = (t-s+1)_+}^{ t} \|\nabla_i f(\xb(\dind+1)) - \nabla_i f(\xb(\dind))\|
	\nonumber\\
	\label{sub_zero_2}
	&{\le}
	\sum_{\dind = (t-s+1)_+}^{t} L \|\xb(\dind+1) - \xb(\dind)\|.
	\end{align}
	Combining the two cases in \cref{sub_zero_1} and \cref{sub_zero_2} we have for all $t$:
	\begin{align}
	\label{subdiff_bound}
	\| \ub(t+1) + \nabla f(\xb(t+1))\| 
	\leq 
	(\sqrt{p}/\eta + 2L) {\sum_{\dind = (t-2s)_+}^t} \| \xb(\dind+1) - \xb(\dind) \|,
	\end{align}
	where $\ub(t+1) = \big( u_1(t+1),~ \ldots~,~ u_p(t+1) \big) \in \partial g (\xb(t+1))$, and the factor $\sqrt{p} \geq 1$ is artificially introduced for the convenience of subsequent analysis. Therefore, by \cref{subdiff_bound} and \Cref{thm:limitcon} we deduce
	\begin{align}
	\label{crit_1}
	\lim_{t\to\infty} \dist_{\partial F(\xb(t+1))}(\zero) \leq  \lim_{t\to\infty } 	\| \ub(t+1) + \nabla f(\xb(t+1))\| = 0.
	\end{align}
	
	Recall that $\xb^*$ is a limit point of $\{\xb(t)\}$, thus there exists a  subsequence $\xb(t_m) \to \xb^*$. Next we verify the function value convergence in \cref{crit_want}. 
	The challenge here is that the component function $g$ is only closed, hence may not be continuous. 
	For any $t \in T_i$, applying \cref{usc} with $z = x_i^*$ and rearranging gives 
	\begin{align}
	g_i(x_i(t + 1)) 
	&\le 
	g_i\big(x^*_i\big) + \frac{1}{2\eta}\|x_i^* - x_i(t)\|^2  - \frac{1}{2\eta}\|x_i(t+1) - x_i(t)\|^2  \nonumber\\
	&\quad + \inner{x_i^* - x_i(t+1)}{\nabla_i f(\xb^i(t))}  \nonumber\\
	\label{uniform}
	&= 
	g_i\big(x^*_i\big) + \frac{1}{2\eta}\|x_i^* - x_i(t)\|^2  - \frac{1}{2\eta}\|x_i(t+1) - x_i(t)\|^2  
	\\
	& \quad + \inner{x_i^* - x_i(t+1)}{\nabla_i f(\xb^*)} + \inner{x_i^* - x_i(t+1)}{\nabla_i f(\xb^i(t)) - \nabla_i f(\xb^*)}.
	\nonumber
	\end{align}
	 Observe from \Cref{thm:limitcon} that 
	\begin{align*}
	\lim_{m\to\infty} \max_{t \in [t_m-s, t_m + s]} \|\xb(t+1) - \xb^*\| = 0, \qquad \lim_{m\to\infty} \max_{t \in [t_m-s, t_m + s]} \|\xb^i(t+1) - \xb^*\| = 0.
	\end{align*}
	By \Cref{assum:skip}, $[t_m -s, t_m + s] \cap T_i \ne \emptyset$ for all $i$. Then using the Lipschitz continuity of $\nabla f$ we deduce from \cref{uniform} that
	\begin{align}
	\label{eq:func_con}
	\limsup_{m \to \infty} \max_{t\in[t_m-s, t_m+s] \cap T_i} g_i(x_i(t+1)) \leq g_i(x_i^*).
	\end{align}
	Since each machine updates at least once during $[t_m-s, t_m]$, let $\hat t_m$ be the largest element of $[t_m-s, t_m] \cap T_i$.
	Also notice that each machine $i$ only updates at its active clocks $T_i$, it then follows that
	\begin{align*}
	\max_{t\in[t_m, t_m+s]} g_i(x_i(t+1)) = \max_{t\in[\hat t_m, t_m+s] \cap T_i} g_i(x_i(t+1)) \leq \max_{t\in[t_m -s, t_m+s] \cap T_i} g_i(x_i(t+1)),
	\end{align*}
	and hence by \eqref{eq:func_con}
	\begin{align}\label{eq: usc}
	\limsup_{m \to \infty} \max_{t\in[t_m, t_m+s]} g_i(x_i(t+1)) \leq  g_i(x_i^*).
	\end{align}
	
	To complete the proof, choose any $k_m\in [t_m, t_m+s]$. Since $\xb(t_m) \to \xb^*$, \Cref{thm:limitcon} implies that 
	\begin{align}
	\label{crit_2}
	\xb(k_m) \to \xb^*.
	\end{align}	
	From \cref{eq: usc} we know for all $i$, $\limsup\limits_{m\to\infty} g_i(x_i(k_m)) \leq g_i(x_i^*)$. On the other hand, it follows from the closedness of the function $g_i$ (\cf \Cref{assum:separable}) that $\liminf\limits_{m\to\infty} g_i(x_i(k_m)) \geq g_i(x_i^*)$, thus in fact $ \lim\limits_{m\to\infty} g_i(x_i(k_m)) = g_i(x_i^*) $. Since $f$ is continuous, we know 
	\begin{align}
	\label{crit_3}
	\lim_{m\to \infty} F(\xb(k_m)) = \lim_{m\to \infty} f(\xb(k_m)) + \sum_i g_i(x_i(k_m)) = F(\xb^*).
	\end{align}
	Combining \cref{crit_1}, \cref{crit_2} and \cref{crit_3} we know  from \Cref{def:sub} that $\xb^* \in \crit{F}$.
\end{proof}

\Cref{thm:limitcon_1} further justifies \mspg by showing that any limit point it produces is necessarily a critical point. Of course, for convex functions any critical point is a global minimizer. The closest result to \Cref{thm:limitcon} and \Cref{thm:limitcon_1} we are aware of is \cite[Proposition 7.5.3]{BT_parallel}, where essentially the same conclusion was reached but under the much more restrictive assumption that $g$ is an indicator function of a product \emph{convex} set. Thus, our result is new even when $g$ is a convex function such as the $\ell_1$ norm that is widely used to promote sparsity. Furthermore, we allow $g$ to be any closed separable function (convex or not), covering the many recent nonconvex regularization functions in machine learning and statistics (see \eg \cite{FanLi01,Mazumderetal11,Zhang10,ZhangZhang12}). We also note that the proof of \Cref{thm:limitcon_1} (for nonconvex $g$) involves significantly new ideas beyond those of \cite{BT_parallel}.
 

We note that the existence of limit points can be guaranteed, for instance, if $\{\xb(t)\}$ is bounded or the sublevel set $\{\xb ~|~ F(\xb) \le \alpha \}$ is bounded for all $\alpha \in \RR$. 
However, we have yet to prove that the sequence $\{ \xb(t) \}$ generated by \mspg does converge to one of the critical points, and we fill this gap under two complementary sets of assumptions on the objective function in \Cref{sec: EB,sec: KL}, respectively.


\section{Convergence under Error Bound}\label{sec: EB}
In this section we prove that the global sequence $\{\xb(t)\}$ produced by \mspg converges periodically linearly to a global minimizer, by assuming an error bound condition on the objective function in \eqref{compmin} and a convexity assumption that serves to simplify the presentation:
\begin{assum}[Convex]\label{assum:convex}
	The functions $f$ and $g$ in \eqref{compmin} are convex.
\end{assum}
Note that for convex functions $g$ the proximal mapping $\prox[\eta]{g}$ is single valued for any $\eta > 0$. The error bound condition we need is as follows:
\begin{assum}[Error Bound]\label{assum:eb}
	For every $\alpha > 0$, there exist $\delta ,\kappa>0$ 
	such that for all $\xb$ with $f(\xb) \leq \alpha$ and $\|\xb - \prox[]{g}(\xb - \nabla f(\xb))\| \leq \delta$,
	\begin{align}\label{eq:eb}
	\dist_{\crit F}(\xb) \leq \kappa \|\xb - \prox[]{g}(\xb - \nabla f(\xb))\|,
	\end{align}
	where recall that $\crit F$ is the set of critical points of $F$.
\end{assum}

\Cref{eq:eb} is a proximal extension of the Luo-Tseng error bound \cite{LuoTseng93} where $g$ is the indicator function of a closed convex set. 
A prototypic convex function $F$ satisfying \eqref{eq:eb} is the following: 
\begin{align}
\label{eq:prototype}
F(\xb) = f(A\xb) + g(\xb),
\end{align}
where $f$ is  strongly convex (\ie, $f-\tfrac{\mu}{2}\|\cdot\|^2$ is convex for some $\mu > 0$), $A$ is a linear map, and $g$ is either an indicator function of a convex set \cite{LuoTseng93} or the $\ell_p$ norm for $p\in [1,2]\cup\{\infty\}$ \cite{EB2015}.  Many machine learning formulations such as Lasso and sparse logistic regression fit into this form. In fact, for convex functions $F$ taking such form, the error bound condition in \cref{eq:eb} is recently shown to be equivalent to the following conditions \cite{DrusvyatskiyLewis16,Zhang2016}:
\begin{align*}
\text{Restricted strong convexity}: ~ &\inner{ \xb - \prox[]{g}(\xb)}{\xb - \proj{\crit F}(\xb)} \ge \mu \cdot \dist_{\crit F}^2(\xb), \label{eq: rsc}\\
\text{Quadratic growth}: ~ &F(\xb) - F^* \ge \mu \cdot \dist_{\crit F}^2(\xb), 
\end{align*}
where $F^*$ is the minimum value of $F$ and $\mu>0$ is a constant. 
In general, the error bound condition in \cref{eq:eb} is not exclusive to convex functions. For instance, it holds for $f(\xb)= \tfrac{1}{2}\|\xb\|^2$ and any function $g$ that has a unique global minimizer at $0$ (such as the cardinality function $g(\xb) = \|\xb\|_0$). However, it is often quite challenging to establish the error bound condition for a large family of nonconvex functions.

We define the following nonnegative quantities that measure the progress of \mspg:
\begin{align}
A(t) &:= F(\xb(t)) - F^*, ~ F^* :=\inf_{\xb} F(\xb), \\
B(t) &:= \sum_{k=(t-s-1)_+}^{t-1} \|\xb(k+1) - \xb(k)\|^2, 
\end{align}
In the following key lemma we relate the gap quantities defined above inductively.
\begin{lemma}\label{lemma: EB_2}
	Let \Cref{assum:bb,assum:smooth,assum:separable,assum:delay,assum:skip,assum:convex,assum:eb} hold. Then, we have 
	\begin{align*}
	& A(t+s+1)  \le A(t) - \tfrac{1}{2}(\tfrac{1}{\eta} - L - 2sL \sqrt{p}) B(t+s+1) + \tfrac{1}{2}sL\sqrt{p} B(t) \\
	& 0 \le A(t+s+1) \le a_\eta B(t+s+1) + b B(t),
	\end{align*}
	where $a_\eta$ and $b$ are given in \eqref{eq:ab} below.
\end{lemma}

\begin{proof}
	The first inequality is obtained by summing the inequality \cref{eq: progress} over $t, t+1, \cdots, t+s$. So we need only prove the second inequality. 

	Let us introduce some notations to simplify the proof.
	For each machine $i$ let $t_i$  be the largest clock in $[t, t+s] \cap T_i$, and denote
	\begin{align}
	\label{eq:z}
	\zb &= \big(x_1(t_1), ~\ldots~, x_p(t_p)\big) \\
	\label{eq:z2} \zb^+ &= \big(x_1(t_1+1), ~\ldots~, x_p(t_p+1)\big) = \big(x_1(t+s+1), ~\ldots~, x_p(t+s+1)\big),
	\end{align}
	where the last equality is due to the maximality of each $t_i$.
	From the optimality condition of the proximal map $z^+_i = \prox{g_i}(z_i - \eta \nabla_i f(\xb^i(t_i)))$ we deduce
	\begin{align}
	\label{eq:subgrad}
	\eta^{-1} (z_i - z_i^+) - \nabla_i f(\xb^i (t_i)) \in \partial g_i (z_i^+).
	\end{align}
	Since the gradient of $f$ is $L$-Lipschitz continuous and the function $g$ is convex, we obtain
	\begin{align*}
	f(\zb^+) - f(\bar\zb) &\le \sum_{i=1}^p \langle z^+_i -  \bar z_i, \nabla_i f(\bar\zb)\rangle + \frac{L}{2} \|\zb^+ - \bar\zb\|^2, \\
	g(\zb^+) - g(\bar\zb) &\le \sum_{i=1}^p \langle z^+_i -  \bar z_i, \eta^{-1} (z_i - z_i^+) - \nabla_i f(\xb^i (t_i)) \rangle,
	\end{align*}
	where we define $\bar \zb := \proj{\crit F}(\zb)$, \ie, the projection of $\zb$ onto the set of critical points of $F$, and the last inequality follows from \cref{eq:subgrad}.
	Adding up the above two inequalities we obtain
	\begin{align*}
	F(&\zb^+) - F^* - \tfrac{L}{2} \|\zb^+ - \bar\zb\|^2 \le \sum_{i=1}^p \langle z^+_i -  \bar z_i, \nabla_i f(\bar\zb) + \eta^{-1} (z_i - z_i^+) - \nabla_i f(\xb^i (t_i))\rangle \\
	&\overset{\mathrm{(i)}}{\le} \sum_{i=1}^{p} [\|z_i^+ - z_i\| + \|z_i - \bar z_i\|] [\|\nabla_i f(\xb^i (t_i)) - \nabla_i f(\bar\zb)\| + \eta^{-1} \|z_i - z_i^+\|] \\	
	&\overset{\mathrm{(ii)}}{\le} \sum_{i=1}^{p} 4\left[ \|z_i^+ - z_i\|^2 + \|z_i - \bar z_i\|^2 + \eta^{-2} \|z_i^+ - z_i\|^2 + \|\nabla_i f(\xb^i (t_i)) - \nabla_i f(\bar\zb)\|^2 \right] \\
	&\le 4\left[ \|\bar\zb - \zb\|^2 + (1+\eta^{-2}) \|\zb^+ - \zb\|^2 +  \sum_{i=1}^{p} L^2 \|\xb^i (t_i) - \bar\zb\|^2 \right],
	\end{align*}
	where (i) is due to the Cauchy-Schwarz inequality and the triangle inequality, (ii) is due to the elementary inequality $(a+b)(c+d) \leq 4(a^2+b^2+c^2+d^2)$, and the last inequality is due to the $L$-Lipschitz continuity of $\nabla f$.
	Using again the triangle inequality we obtain from the above inequality that
	\begin{align}
	F(\zb^+) - F^* &\le (L+4)\|\bar\zb - \zb\|^2 + (L+4+\tfrac{4}{\eta^{2}}) \|\zb^+ - \zb\|^2 + 4L^2\sum_{i=1}^{p} \|\xb^i (t_i) - \zb\|^2 \nonumber\\
	\nonumber&\overset{\mathrm{(i)}}{=} (L\!+\!4)\|\bar\zb \!-\! \zb\|^2 \!+\! \sum_{i=1}^{p} [(L\!+\!4\!+\!\tfrac{4}{\eta^{2}})\|x_i(t_i\!+\!1) \!-\! x_i(t_i)\|^2 \!+\! 4L^2\|\xb^i (t_i) \!-\! \zb\|^2],\\
	&\overset{\mathrm{(ii)}}{\leq} (L\!+\!4)\|\bar\zb \!-\! \zb\|^2 + (L\!+\!4\!+\!\tfrac{4}{\eta^{2}}) B(t+s+1) + 4L^2\sum_{i=1}^{p}  \|\xb^i (t_i) \!-\! \zb\|^2,	
	\label{eq:Fgap}
	\end{align}
	where (i) is due to our definition of $\zb$ and $\zb^+$ in \eqref{eq:z} and \eqref{eq:z2}, and (ii) is due to the fact that $t_i \in [t, t+s]$ for all $i$.
	
	We next bound the terms $\|\bar\zb - \zb\|^2$ and $\|\xb^i (t_i) \!-\! \zb\|^2$. First, note that
	\begin{align*}
	\| x_i(t_i+1) - x_i(t_i)\| &= \|\prox{g_i}(x_i(t_i) - \eta \nabla_i f(\xb^i (t_i))) - x_i(t_i)\| \\
	&\ge \|\prox{g_i}(x_i(t_i) - \eta \nabla_i f(\zb)) - x_i(t_i)\| \\
	&\quad - \|\prox{g_i}(x_i(t_i) - \eta \nabla_i f(\xb^i (t_i))) - \prox{g_i}(x_i(t_i) - \eta \nabla_i f(\zb)) \| \\
	&\overset{\mathrm{(i)}}{\ge} \|\prox{g_i}(x_i(t_i) - \eta \nabla_i f(\zb)) - x_i(t_i)\| - \eta L \|\zb - \xb^i(t_i)\|,
	\end{align*}
	where (i) follows from the non-expansiveness of $\prox{g}$ (recall that $g$ is convex) and the $L$-Lipschitz continuity of $\nabla f$.
	Rearranging the above inequality and summing over all $i$, we obtain
	\begin{align}
	\nonumber \|\prox{g}(\zb- \eta \nabla f(\zb)) - \zb\|^2 &\le \sum_{i=1}^p \left[\| x_i(t_i+1) - x_i(t_i)\| +  \eta L \|\zb - \xb^i(t_i)\| \right]^2 \\
	\label{eq:proxb}&\le 2\sum_{i=1}^p \left[\| x_i(t_i+1) - x_i(t_i)\||^2 +  \eta^2 L^2 \|\zb - \xb^i(t_i)\|^2 \right]. 
	\end{align}
	The last term $\|\zb - \xb^i(t_i)\|^2$ can be further bounded as follows:  
	\begin{align}
	\nonumber
	\|\zb - \xb^i(t_i)\|^2 &= \sum_{j=1}^p \|x_j(t_j) - x_j (\tau_j^i (t_i))\|^2 \\
	&= \sum_{j=1}^p \Big\|\sum_{k = \min\{ t_j, \tau_j^i(t_i) \} }^{\max\{ t_j, \tau_j^i(t_i) \}-1} x_j(k+1) - x_j(k) \Big\|^2 \nonumber\\
	&\le \sum_{j=1}^p \Big[\sum_{k = \min\{ t_j, \tau_j^i(t_i) \} }^{\max\{ t_j, \tau_j^i(t_i) \}-1} \| x_j(k+1) - x_j(k) \| \Big]^2 \nonumber\\	
	&\overset{\mathrm{(i)}}{\le} \sum_{j=1}^p 2s \sum_{k= t-s}^{t+s-1} \| x_j(k+1) - x_j(k) \|^2 \nonumber\\
	\nonumber &= 2s \sum_{k=t-s}^{t+s-1} \|\xb(k+1) - \xb(k)\|^2 \\
	\label{eq:diff} & \leq 2s [B(t) + B(t+s+1)],
	\end{align}
	where (i) is due to the fact that $t_j \in [t, t+s]$ and $\tau_j^i(t_i) \in [t-s, t+s]$.
	Combining \eqref{eq:proxb} and \eqref{eq:diff} we obtain
	\begin{align}
	\label{eq:pb} \|\prox{g}(\zb - \eta \nabla f(\zb)) &- \zb\|^2 
	\leq 2 B(t+s+1) +  4ps\eta^2 L^2 [B(t) + B(t+s+1)].
	\end{align}
	Thanks to \Cref{thm:limitcon}, we know for $t$ sufficiently large, $\|\prox{g}(\zb \!-\! \eta \nabla f(\zb)) \!-\! \zb\| \leq \eta\delta$. Since the function $\eta \mapsto \frac{1}{\eta}\|\prox{g}(\zb \!-\! \eta \nabla f(\zb)) \!-\! \zb\|$ is monotonically decreasing \cite{Sra12}, we can apply the error bound condition in \Cref{assum:eb} for $\eta < 1$ and $t$ sufficiently large, and obtain
	\begin{align}
	\label{eq:projdist}
	\|\bar\zb - \zb\|^2 \le \kappa \|\zb - \prox[]{g}(\zb - \nabla f(\zb))\|^2 \le \kappa \eta^{-2} \|\zb - \prox{g}(\zb - \eta\nabla f(\zb))\|^2.
	\end{align}
	
	Finally, combining \eqref{eq:Fgap}, \eqref{eq:diff}, \eqref{eq:pb} and \eqref{eq:projdist} we arrive at:
	\begin{align}
	\nonumber
	F(\xb(t\!+\!s\!+\!1)) - F^* &= F(\zb^+) - F^* \\
	\nonumber&\le (L\!+\!4)\|\bar\zb \!-\! \zb\|^2 + (L\!+\!4\!+\!\tfrac{4}{\eta^{2}}) B(t\!+\!s\!+\!1) + 4L^2\sum_{i=1}^{p}  \|\xb^i (t_i) \!-\! \zb\|^2,\\	
	&\le  a_\eta B(t+s+1) +  b B(t),
	\end{align}
	where the coefficients are
	\begin{align}
	\label{eq:ab}
	a_\eta &= L+4+8psL^2 + 4ps\kappa L^2(L+4) + \tfrac{2}{\eta^2}(2+4\kappa+\kappa L), \\
	b &= 8psL^2+4ps\kappa L^2(L+4).
	\end{align}

\end{proof}

\Cref{lemma: EB_2} improves the analysis of \cite{Tseng_linear_PAA} in three aspects: (1) it is shorter and simpler; (2) it allows any convex function $g$; and (3) the leading coefficient for $B(t)$ is reduced from $O(1/\eta)$ to $O(1)$. 
The two recursive relations in Lemma \ref{lemma: EB_2}, as shown in \cite[Lemma 4.5]{Tseng_linear_PAA}, easily imply the following convergence guarantee:
\begin{theorem}\label{coro: s step linear}
	Let \Cref{assum:bb,assum:smooth,assum:separable,assum:delay,assum:skip,assum:convex,assum:eb} hold. Then, there exists some $\eta_0 >0$ such that if $0<\eta<\eta_0$, then the sequences $\{A(t), B(t)\}$ generated by \mspg satisfy for all $r = 0, 1, 2, \cdots$
	\begin{align}
	A(r (s+1)) \le C_1 (1-\gamma \eta)^{r},  ~B(r (s+1)) \le C_2 (1-\gamma \eta)^{r},
	\end{align}
	where $C_1, C_2,\gamma<1/\eta$ are positive constants.
\end{theorem}

Hence, the gaps $A(t)$ and $B(t)$ that measure the progress of \mspg decrease by a constant factor $(1-\gamma \eta)$ for every $s+1$ steps, which makes intuitive sense since in the worst case each worker machine only performs one update in every $s+1$ steps. In other words, $(s+1)$ is the natural time scale for measuring progress here. Note that since $\|\xb(t+s+1) - \xb(t)\|^2 \leq (s+1) B(t+s+1)$, it follows easily that the global sequence $\xb(t)$ and consequently also the local sequences $\{ \xb^i(t) \}$ all converge to the same limit point in $\crit F$ at a $(s+1)$-periodically linear rate.


\section{Convergence with \KL inequality}\label{sec: KL}
The error bound condition considered in the previous section is not easy to verify in general. It has been discovered recently that the error bound condition is equivalent to other notions in optimization that can be verified in alternative ways \cite{DrusvyatskiyLewis16,Zhang2016}, see \eg \eqref{eq:prototype}. However, for nonconvex functions, sometimes even the simple ones, it remains a challenging task to verify if the error bound condition holds. This failure motivates us to investigate another property, the Kurdyka-${\L}$ojasiewicz (\KL) inequality, that has been shown to be quite effective in dealing with nonconvex functions.

\begin{definition}[\KL property, {\cite[Lemma 6]{prox_alter_linear}}]\label{def:KL}
	Let $\Omega\subset \mathrm{dom} h$ be a compact set on which the function $h$ is a constant. We say that $h$ satisfies the \KL property if there exist $\varepsilon, \lambda >0$ such that for all $\bar{\xb} \in \Omega$ and all $\xb\in \{\zb\in \RR^d : \dist_\Omega(\zb)<\varepsilon\}\cap [\zb: h(\bar{\xb}) < h(\zb) <h(\bar{\xb}) + \lambda]$, it holds that
	\begin{align}\label{KLineq}
	\varphi'(h(\xb) - h(\bar{\xb})) \cdot \dist_{\partial h(\xb)}(\zero) \ge 1,
	\end{align}
	where the function $\varphi: [0,\lambda) \to \RR_+, 0\mapsto 0$, is continuous, concave, and has continuous and positive derivative $\varphi'$ on $(0,\lambda)$.
\end{definition}

The \KL inequality in \cref{KLineq} is an important tool to bound the trajectory length of a dynamical system (see \cite{BolteDLM10,initial_KL} and the references therein for some historic developments). It has recently been used to analyze discrete-time algorithms in \cite{AbsilMA05} and proximal algorithms in \cite{KL_gene_rate,KL_bound_sequence,prox_alter_linear}.
As we shall see, the function $\varphi$ will serve as a Lyapunov potential function.
Quite conveniently, most practical functions, in particular, the quasi-norm $\|\cdot\|_p$ for positive rational $p$, as well as convex functions with certain growth conditions, are \KL. 
For a more detailed discussion of \KL functions, including many familiar examples, see \cite[Section 5]{prox_alter_linear} and \cite[Section 4]{KL_bound_sequence}.

Following the recipe in \cite{prox_alter_linear}, we need the following assumption to guarantee the algorithm is making \emph{sufficient} progress:
\begin{assum}[Sufficient decrease]\label{assum:sd}
	There exists $\alpha \!>\! 0$ such that for all large $t$,
	\begin{align}
	F(\xb(t+1)) \leq F(\xb(t)) - \alpha \|\xb(t+1) - \xb(t)\|^2.
	\end{align}
\end{assum}
The sufficient decrease assumption is automatically satisfied in many descent algorithms, e.g., the proximal gradient algorithm. However, in the partially asynchronous parallel (\pap) setting, it is highly nontrivial to satisfy the sufficient decrease assumption because of the complication due to communication delays and update skips. 
Note also that none of the worker machines actually has access to the global sequence $\xb(t)$, so even verifying the sufficient decrease property is not trivial.
To simplify the presentation, we first analyze the performance of \mspg using the \KL inequality and taking the sufficient decrease property for granted, and later we we will give some verifiable conditions to justify this simplification.  

Our first result in this section strengthens the convergence properties in \Cref{thm:limitcon,thm:limitcon_1} for \mspg:
\begin{theorem}[Finite Length]\label{thm:finite}
	Let \Cref{assum:bb,assum:smooth,assum:separable,assum:delay,assum:skip,assum:sd} hold for \mspg,  and let $F$ satisfy the \KL property in \Cref{def:KL}. Then, with step size $\eta \in \left(0, \frac{1}{L(1+2\sqrt{p}s)}\right)$, every bounded sequence $\{\xb(t)\}$ generated by \mspg satisfies
	\begin{flalign}
	\label{global_finite}
	\sum_{t=0}^{\infty} \|\xb(t+1) - \xb(t)\| < \infty,\\
	\label{local_finite} 
	\forall i = 1, \ldots ,p, ~ 
	\sum_{t=0}^{\infty} \|\xb^i(t+1) - \xb^i(t)\| < \infty.
	\end{flalign}
	Furthermore, $\{\xb(t)\}$ and $\{\xb^i(t)\}_{i=1}^p$ converge to the same critical point of $F$.
\end{theorem}

\begin{proof}
	We first show that \cref{global_finite} implies \cref{local_finite}. Indeed, recall from \eqref{eq:local_diff}:
	\begin{align*}
	\|\xb^i(t+1) - \xb^i(t)\| 
	&\leq 
	\sum_{k=(t-s)_+}^{t} \|\xb(k+1) - \xb(k)\|.
	\end{align*}
	Therefore, summing for $t=0, 1, \cdots, n$ gives
	\begin{align*}
	\sum_{t=0}^n \|\xb^i(t+1) - \xb^i(t)\| 
	&\leq 
	\sum_{t=0}^n  \sum_{k=(t-s)_+}^{t} \|\xb(k+1) - \xb(k)\|
	\\
	&\leq 
	(2s+1) \sum_{t=0}^n \|\xb(t+1) - \xb(t)\|.
	\end{align*}
	The claim then follows by letting $n$ tend to infinity. 
	
	By \Cref{thm:limitcon}, the limit points of $\{\xb(t)\}$ and $\{\xb^i(t)\}_{i=1}^p$ coincide and are critical points of $F$. Thus, the only thing left to prove is the finite length property in \cref{global_finite}. 
	By \Cref{assum:sd} and \Cref{assum:bb}, the objective value $F(\xb(t))$   decreases to a finite limit $F^*$. 
	Since $\{\xb(t)\}$ is assumed to be bounded, the set of its limit points $\Omega$ is nonempty and compact. Summing \cref{usc} over all $i$ and set $\zb \in \Omega$, we obtain
	\[
	g(\xb(t+1)) \le g(\zb) - \frac{1}{2\eta} \|\xb(t+1) - \xb(t)\|^2 - \sum_{i=1}^{p}\inner{\nabla_i f(\xb^i(t))}{\xb(t+1) - \xb(t)}.
	\]
	Note that $\xb(t+1) - \xb(t) \to 0$.  Also, since $\{\xb(t)\}$ is bounded and $\xb(t) - \xb^i(t)\to 0$ for all $i$, $\{\xb^i(t)\}_{i=1}^p$ are all bounded.  we then take limsup on both sides and obtain that $\limsup_{t\to\infty} g(\xb(t+1)) \le g(\zb)$. Together with the closedness of $g$ we further obtain that $\lim_{t\to\infty} g(\xb(t+1)) = g(\zb)$. Note that $f$ is continuous, we thus conclude that $\lim_{t\to\infty} F(\xb(t+1)) = F(\zb)$ for all $\zb\in \Omega$. Note that $F(\xb(t)) \downarrow F^*$. Thus for all $\xb^* \in \Omega$, we have $F(\xb^*) \equiv F^*$. 
	Now fix $\varepsilon > 0$. Since $\Omega$ is compact, for $t$ sufficiently large we have $\dist_{\Omega}(\xb(t)) \leq \varepsilon$. We now have all ingredients to apply the  \KL inequality in \Cref{def:KL}: for all sufficiently large $t$,
	\begin{flalign}\label{uKL}
	\varphi'\big(F(\xb(t)) - F^*\big) \cdot \dist_{\partial F(\xb(t))}(\zero) \ge 1.
	\end{flalign}
	Since $\varphi$ is concave, we obtain
	\begin{flalign}
	\Delta_{t,t+1} &:= \varphi\big(F(\xb(t)) - F^*\big) - \varphi\big(F(\xb(t+1)) - F^*\big) \nonumber\\
	&\ge \varphi'\big(F(\xb(t)) - F^*\big) \big(F(\xb(t)) -F(\xb(t+1))\big)
	\nonumber\\
	\label{Lyapunov}
	&\ovset{(i)}{\ge} 
	\frac{\alpha \|\xb(t+1) - \xb(t) \|^2}{\dist_{\partial F(\xb(t))}(\zero)},
	\end{flalign}
	where (i) follows from \Cref{assum:sd} and \cref{uKL}.
	It is clear that the function $\varphi$ (composed with $F$) serves as a Lyapunov function. Using the elementary inequality $2\sqrt{ab} \le a + b$ we obtain from \cref{Lyapunov} that for $t$ sufficiently large,
	\begin{flalign*}
	2\|\xb(t+1) - \xb(t)\| 
	\le 
	\tfrac{\delta}{\alpha} \Delta_{t,t+1} + \tfrac{1}{\delta}\dist_{\partial F(\xb(t))}(\zero),
	\end{flalign*}
	where $\delta > 0$ will be specified later. Recalling the bound for $\partial F(\xb(t))$ in \cref{subdiff_bound}, and summing over $t$ from $m$ (sufficiently large) to $n$ gives:
	\begin{flalign*}
	2\sum_{t=m}^{n}&\|\xb(t+1) - \xb(t)\| 
	\le  
	\sum_{t=m}^{n} \frac{\delta}{\alpha} \Delta_{t, t+1} + 
	\sum_{t=m}^{n} \frac{1}{\delta} \dist_{\partial F(\xb(t))}(\zero)
	\\
	&\ovset{(i)}{\le}
	\frac{\delta}{\alpha} \varphi\big(F(\xb(m)) - F^*\big) + 
	\sum_{t=m}^{n} \frac{\sqrt{p}/\eta + 2L}{\delta} \sum_{\dind = (t-2s)_+}^t \| \xb(\dind+1) - \xb(\dind) \|
	\\
	&\le
	\frac{\delta}{\alpha} \varphi\big(F(\xb(m)) - F^*\big) + 
	\frac{(2s+1)(\sqrt{p}/\eta + 2L)}{\delta}
	\sum_{\dind=(m-2s)_+}^{m-1} \| \xb(\dind+1) - \xb(\dind) \|
	\\
	&\qquad\qquad \qquad\qquad + 
	\frac{(2s+1)(\sqrt{p}/\eta + 2L)}{\delta}
	\sum_{t=m}^{n} \| \xb(t+1) - \xb(t) \|,
	\end{flalign*}
	where (i) is due to \cref{subdiff_bound}.
	Setting $\delta = (2s+1)(\sqrt{p}/\eta + 2L)$ and rearranging gives 
	\begin{flalign*}
	\sum_{t=m}^{n}\|\xb(t+1) - \xb(t)\| 
	&\le 
	\frac{(2s+1)(\sqrt{p}/\eta + 2L)}{\alpha} \varphi\big(F(\xb(m)) - F^*\big) \\
	&\quad+ 
	\sum_{\dind=(m-2s)_+}^{m-1} \| \xb(\dind+1) - \xb(\dind) \|.
	\end{flalign*}
	Since the right-hand side is finite, let $n$ tend to infinity completes the proof for \cref{global_finite}.
\end{proof}

Compared with \eqref{sq_summable}  in \Cref{thm:limitcon}, we now have the successive differences to be absolutely summable (instead of square summable). This is a significantly stronger result as it immediately implies that the whole sequence is Cauchy and hence convergent, whereas we cannot get the same conclusion from the square summable property in \Cref{thm:limitcon}.
We note that local maxima are excluded from being the limit in \Cref{thm:finite}, due to \Cref{assum:sd}. Also, the boundedness assumption on the trajectory $\{\xb(t)\}$ is easy to satisfy, for instance, when $F$ has bounded sublevel sets. We refer to \cite[Remark 3.3]{KL_bound_sequence} for more  conditions that imply the boundedness condition. Moreover, following similar arguments in \cite{KL_bound_sequence} we can also determine the local convergence rates of the sequences generated by \mspg.

In the remaining part of this section we provide some justifications for the sufficient decrease property in \Cref{assum:sd}.  For simplicity we assume all worker machines perform updates in each time step $t$: 
\begin{assum}\label{assum:always}
	$\forall i = 1, \cdots, p, \forall t, t \in T_i.$
\end{assum}
Note that \Cref{assum:always} is commonly adopted in the analysis of many recent parallel systems \cite{distrib_delay,delay_pga_linear,ssp_parameter,parameter_server,asy_coord_prox,hogwild}. Put it differently, under \Cref{assum:always} we measure the performance of the system \wrt the minimum number of updates among all worker machines whereas  under the more relaxed  \Cref{assum:skip} we measure the performance \wrt the total number of updates among all worker machines.

We will replace the sufficient decrease property in \Cref{assum:sd} with the following key property that turns out to be easier to verify:
\begin{assum}[Proximal Lipschitz]\label{assum:EL}
	We say a pair of functions $f$ and $g$ satisfy the proximal Lipschitz property on a sequence $\{\xb(t)\}$ if for all $\eta$ sufficiently small, there exists $L_\eta \in o(1)$, \ie $L_\eta \to 0$ as $\eta \to 0$, such that for all large $t$,
	\begin{align}
	\| \Delta_\eta(\xb(t)) - \Delta_\eta(\xb(t+1)) \| \le L_\eta \|\xb(t) - \xb(t+1)\|,
	\end{align}
	where\footnote{Should the proximal map be multi-valued, we contend with any single-valued selection.}
	$\Delta_\eta(\xb) \in \prox{g}(\xb - \eta\nabla f(\xb)) - \xb$.
\end{assum}

The proximal Lipschitz assumption is motivated by the special case where $g \equiv 0$ and hence $\Delta_\eta(\xb) = -\eta\nabla f(\xb)$ is $\eta$-Lipschitz, thanks to \Cref{assum:smooth}. As we have seen in previous sections, Lipschitz continuity plays a crucial role in our proof where a major difficulty is to control the inconsistencies among different worker machines due to communication delays. Similarly here, the proximal Lipschitz property, as we show next, allows us to remove the sufficient decrease property in \Cref{assum:sd}---the seemingly strong assumption that we needed in proving our main result \Cref{thm:finite}. 

Let us first present a quick justification for \Cref{assum:EL}.
\begin{lemma}\label{thm:lcg}
Suppose the functions $f$ and $g$ both have Lipschitz continuous gradient, then \Cref{assum:EL} holds for any sequence $\{\xb(t)\}$.
\end{lemma}
\begin{proof}
Let us denote $L_f$ and $L_g$ as the Lipschitz constant of the gradient $\nabla f$ and $\nabla g$, respectively. Since $\Delta_\eta(\xb) \in \prox{g}(\xb - \eta \nabla f(\xb)) - \xb$, using the optimality condition for the proximal map, see for instance \cite[Proposition 7(iii)]{YuZMX15}, we have
$$
\xb + \Delta_\eta(\xb) + \eta \nabla g\big(\xb + \Delta_\eta(\xb)\big) = \xb - \eta \nabla f(\xb),
$$
and similarly
$$
\zb + \Delta_\eta(\zb) + \eta \nabla g\big(\zb + \Delta_\eta(\zb)\big) = \zb - \eta \nabla f(\zb).
$$
Subtracting one inequality from another, we obtain
\begin{align*}
\|\Delta_\eta(\xb) - \Delta_\eta(\zb)\|  &=  \|\eta \nabla g\big(\zb + \Delta_\eta(\zb)\big) - \eta \nabla g\big(\xb + \Delta_\eta(\xb)\big) + \eta \nabla f(\zb) - \eta \nabla f(\xb)\|
\\
&\leq
\eta L_g \| \zb - \xb + \Delta_\eta(\zb) - \Delta_\eta(\xb) \| + \eta L_f \|\zb - \xb\|
\\
&\leq
\eta L_g \| \Delta_\eta(\zb) - \Delta_\eta(\xb) \| + \eta (L_f+L_g) \|\zb - \xb\|.
\end{align*}
Rearranging we obtain
$$
\|\Delta_\eta(\xb) - \Delta_\eta(\zb)\| \leq \frac{\eta(L_f+L_g)}{1-\eta L_g} \|\zb-\xb\|,
$$
when $0 < \eta < 1/L_g$. Clearly, when $\eta$ is mall, the leading coefficient $\tfrac{\eta(L_f+L_g)}{1-\eta L_g} \in \mathcal{O}(\eta) \subseteq o(1)$, and our proof is complete.
\end{proof}
It is clear that \Cref{thm:lcg} captures the motivating case $g\equiv 0$, but also many other important functions, such as the widely-used regularization function $g = \|\cdot\|_p^p$ for any $p > 1$.
We can now continue with our next result in this section. 

\begin{theorem}\label{thm:main}
	Let \Cref{assum:bb,assum:smooth,assum:separable,assum:delay,assum:always} hold for \mspg, and let $F$ satisfy the \KL property in \Cref{def:KL}. Fix any $r > 1$ with $C = \tfrac{r^{s+1}-1}{r-1}$ and step size $\eta$ such that 
	$\eta < \frac{1}{L(1+2\sqrt{p}C+2\sqrt{p}s)}$. If for each local sequence $\{\xb^i(t)\}$ generated by \mspg, \Cref{assum:EL} holds with $L_\eta \leq \frac{r^2-1}{2pr^2C^2}$, and the global sequence $\{\xb(t)\}$ is bounded, then the finite length properties in \eqref{global_finite} and \eqref{local_finite} hold. In particular, $\{\xb(t)\}$ and $\{\xb^i(t)\}_{i=1}^p$ converge to the same critical point of $F$.
\end{theorem}
\begin{proof}
%
%
%
%
Using the elementary inequality $\|a\|^2 - \|b\|^2 \le 2\|a\|\|a - b\|$, we have for all $t$:
\begin{flalign}
\|\xb(t+1) - &\xb(t)\|^2 - \left\|\xb(t+2) - \xb(t+1)\right\|^2 \nonumber\\
&\le 2\left\|\xb(t+1) - \xb(t)\right\| \cdot \left\|(\xb(t+1) - \xb(t)) - (\xb(t+2) - \xb(t+1))\right\|
 \nonumber\\
 &\leq  2\left\|\xb(t+1) - \xb(t)\right\| \cdot \sum_{i=1}^{p}\left\|(x_i(t+1) - x_i(t)) - (x_i(t+2) - x_i(t+1))\right\|\nonumber\\
&\ovset{(i)}{\leq}  2\left\|\xb(t+1) - \xb(t)\right\| \cdot \sum_{i=1}^{p} \left \|\Delta_\eta(\xb^{i} (t)) - \Delta_\eta(\xb^{i} (t+1))\right\|\nonumber\\
&\ovset{(ii)}{\le} 2\left\|\xb(t+1) - \xb(t)\right\| \left(\sum_{i=1}^{p} L_\eta \|\xb^{i} (t) - \xb^{i} (t+1)\| \right)\nonumber\\
&\ovset{(iii)}{\le}  2pL_\eta \left\|\xb(t+1) - \xb(t)\right\|  \cdot \sum_{k=(t-s)_+}^{t}\|\xb(k+1) - \xb(k)\|,  \label{eq:induction}
\end{flalign}
where (i) is due to \Cref{assum:always} hence $t\in T_i$ for all $t$, (ii) follows from \Cref{assum:EL}, and (iii) is due to \eqref{eq:local_diff}.

If for some $r > 1$ there exists some $T$ such that for all $t \geq T$,
\begin{align}
\label{eq:case}
\sum_{k=( t-s)_+}^{ t}\|\xb(k+1) - \xb(k)\| \ge C \|\xb(t+1) - \xb(t)\|,
\end{align}
where $C=\frac{r^{s+1}-1}{r-1} > s+1$ (since $r > 1$ and w.l.o.g. $s > 0$). Summing the index $t$ from $T$ to $n$ yields
\begin{flalign}
C\sum_{t= T}^{n}\left\|\xb(t+1) - \xb(t)\right\| &\leq \sum_{t=T}^{n}\sum_{k=(t-s)_+}^{t}\left\|\xb(k+1) - \xb(k)\right\|\nonumber\\
&\leq (s+1)\sum_{t=(T-s)_+}^{n}\left\|\xb(t+1) - \xb(t)\right\|,\nonumber
\end{flalign}
which after rearranging terms becomes
\begin{flalign}
(C -  s - 1)\sum_{t= T}^{n}\left\|\xb(t+1) - \xb(t)\right\| \leq (s+1)\sum_{t=(T-s)_+}^{T-1}\left\|\xb(t+1) - \xb(t)\right\|. \nonumber
\end{flalign}
Since the right hand side does not depend on $n$, letting $n$ tend to infinity we conclude
\begin{flalign}\label{eq: case2}
\sum_{t=0}^{\infty}\|\xb(t+1) - \xb(t)\|< \infty,
\end{flalign}
and the proof of the finite length property would be complete.


Therefore, in the remaining part of the proof, we can assume \eqref{eq:case} fails for infinitely many $t$. Take any such $t=\hat t$, we have 
\begin{align}
\label{eq:induction2}
\sum_{k=( t-s)_+}^{ t}\|\xb(k+1) - \xb(k)\| \leq C \|\xb(t+1) - \xb(t)\| \leq C^2 \|\xb(t+1) - \xb(t)\|,
\end{align}
since $C > 1$. 
Combining \eqref{eq:induction} and \eqref{eq:induction2} we have for $t=\hat t$:
\begin{flalign}
\left\|\xb(t+1) - \xb(t)\right\|^2 - \left\|\xb(t+2) - \xb(t+1)\right\|^2 
&\le 2pL_\eta C^2 \left\|\xb(t+1) - \xb(t)\right\|^2 \nonumber\\
&\leq \left(1-\frac{1}{r^2} \right) \left\|\xb(t+1) - \xb(t)\right\|^2,\nonumber
\end{flalign}
if $\eta$ is small enough (recall that $L_\eta = o(1)$).
After rearranging terms we conclude that for $t=\hat t$:
\begin{align}
\label{eq:ind}
\|\xb(t+1) - \xb(t)\| \le r\|\xb(t+2) - \xb(t+1)\|.
\end{align}
Using induction we can continue the same process for any $t \geq \hat t$. Indeed, suppose \eqref{eq:ind} is true for any $t \leq m-1$, then \eqref{eq:induction} holds (for any $t$), and \eqref{eq:induction2} also holds: If $m \leq \hat t + s$, then
\begin{align*}
\sum_{k=(m-s)_+}^{m}\!\!\! \|\xb(k\!+\!1) \!-\! \xb(k)\| 
&=
\sum_{k=(m-s)_+}^{\hat t}\|\xb(k+1) - \xb(k)\| + \sum_{k=\hat t+1}^{m}\|\xb(k+1) - \xb(k)\| 
\\
&\ovset{(i)}{\leq} 
\sum_{k=(\hat t-s)_+}^{\hat t} \!\! \|\xb(k\!+\!1) \!-\! \xb(k)\| + \sum_{k=\hat t+1}^{m} \!\!r^{m-k} \|\xb(m\!+\!1) \!-\! \xb(m)\| 
\\
&\ovset{(ii)}{\leq}  
C\left[\|\xb(\hat t+1) - \xb(\hat t)\| + \sum_{k=\hat t+1}^{m} r^{m-k} \|\xb(m+1) - \xb(m)\|\right] 
\\
&\ovset{(iii)}{\leq} 
C \sum_{k=\hat t}^{m} r^{m-k} \|\xb(m+1) - \xb(m)\| \\
&\ovset{(iv)}{\leq} 
C^2 \|\xb(m+1) - \xb(m)\|,
\end{align*}
where (i) is due to the induction hypothesis, (ii) is due to the definition of $\hat t$ and the fact that $C>1$, (iii) is due to again the induction hypothesis, and finally (iv) is due to the definition of $C$ (recall $m\leq \hat t +s$). If $m > \hat t + s$, the same inequality, with $C^2$ replaced by $C$, would still hold (essentially dropping all the first terms on the right hand side of the above inequalities). Thus, \eqref{eq:induction} and \eqref{eq:induction2} would imply again \eqref{eq:ind} for $t = m$.

Lastly, we recall from \cref{eq: func_moment} that for large $t$, 
\begin{flalign}
F\big(\xb(t+1)\big) - F\big(\xb(t)\big) 
&\le 
\tfrac{1}{2}(L - 1/\eta) \|\xb(t+1) - \xb(t)\|^2 \nonumber\\
&\quad+ \sqrt{p}L\|\xb(t+1)-\xb(t)\| \sum_{k=(t-s)_+}^{t-1} \|\xb(k+1) - \xb(k) \|. \nonumber
\\
&\leq
\tfrac{1}{2}(L - 1/\eta) \|\xb(t+1) - \xb(t)\|^2  + \sqrt{p}CL\|\xb(t+1)-\xb(t)\|^2. \nonumber
\\
&\leq
-\alpha \|\xb(t+1) - \xb(t)\|^2, \nonumber
\end{flalign}
where $\alpha = \tfrac{1}{2}(1/\eta - L - 2\sqrt{p}CL) > 0$ if $\eta$ is small. Hence, the sufficient decrease property in \Cref{assum:sd} is verified and the finite length properties follow from \Cref{thm:finite}.
\end{proof}

Lastly, we show that \Cref{assum:EL} also holds for the important cardinality function $\|\xb\|_0$ (number of nonzero entries).
\begin{lemma}\label{thm:example}
	Consider the same setting as in \Cref{thm:limitcon}, 
	then \Cref{assum:EL} holds for any function $f$ and $g = \|\cdot\|_0$ on all local sequences $\{\xb^i(t)\}$ of \mspg.	
	\end{lemma}
\begin{proof}



The crucial observation here is that for the cardinality function $g = \|\cdot\|_0$, its proximal map on the $j$-th entry can be chosen as:
\begin{flalign}
\label{eq:prox_l0}
\mathrm{prox}_{g_j}^{\eta}(z_j) =  
\left\{ \begin{array}{l}
z_j, ~\text{if $|z_j|>\sqrt{2\eta}$}\\
0, ~\text{otherwise}
\end{array} \right.
.
\end{flalign}
However, \Cref{thm:limitcon} implies that $\lim\limits_{t\to\infty} \|\xb^i(t+1) - \xb^i(t)\| = 0$. Thus, for $t$ sufficiently large, the sequence $\{\xb^i(t)\}$ will have the same support $\Omega$ (indices that have nonzero entries), for otherwise $\|\xb^i(t+1) - \xb^i(t)\| \geq \sqrt{2\eta}$ even if one index in the support changes.
Therefore,
\begin{align*}
\|\Delta_\eta(\xb^i(t+1)) - \Delta_\eta(\xb^i(t))\| 
&\ovset{(i)}{\leq}
\sum_{j\in \Omega}\|\prox{g_j}(x_j^i(t+1) - \eta \nabla_j f(\xb^i(t+1))) - x_j^i(t+1)\nonumber\\
&\qquad\qquad\qquad - \prox{g_j}(x_j^i(t) - \eta \nabla_j f(\xb^i(t))) - x_j^i(t)\|
\\
&\ovset{(ii)}{\leq}
\sum_{j\in \Omega} \|\eta \nabla_j f(\xb^i(t+1)) - \eta \nabla_j f(\xb^i(t)) \|
\\
&\ovset{(iii)}{\leq}
\eta p L \|\xb^i(t+1) - \xb^i(t)\|,
\end{align*}
where (i) is the triangle inequality, (ii) uses the property of the proximal map \eqref{eq:prox_l0}, and (iii) is due to \Cref{assum:smooth}.
\end{proof}

Note that similar results as \Cref{thm:example} can be derived for the rank function, and more generally for functions whose proximal map is discontinuous with pieces satisfying \Cref{thm:lcg}.


\section*{Acknowledgment}
This work is supported in part by the grants AFOSR FA9550-16-1-0077 and NSF ECCS 16-09916.

\section{Conclusion}\label{sec: conclude}
We have proposed \mspg as an extension of the proximal gradient algorithm to the model parallel and partially asynchronous setting. 
\mspg allows worker machines to operate asynchronously as long as they are not too far apart, hence greatly improves the system throughput.
The convergence properties of \mspg are thoroughly analyzed. In particular, we proved that: 1) every limit point of the sequences generated by \mspg is a critical point of the objective function; 2) under an additional error bound condition, the function values decay periodically linearly; 3) under the additional Kurdyka-${\L}$ojasiewicz inequality, the sequences generated by \mspg converge to the same critical point, provided that a proximal Lipschitz condition is satisfied.
In the future we plan to further weaken the proximal Lipschitz condition so that our analysis can handle many more nonsmooth functions.



\bibliographystyle{siamplain}
\bibliography{./ref}
\end{document}